\documentclass[
leqno, 
12pt]{amsart}
\usepackage{amssymb,amsmath, amsthm,latexsym, mathrsfs}

\usepackage{mathtools}

\usepackage{a4wide}

\usepackage{hyperref}

\usepackage[all]{xy}
\usepackage{rotating}

\usepackage{color}

\theoremstyle{plain}

\newtheorem{lemma}{Lemma}[section]
\newtheorem{theorem}[lemma]{Theorem}

\newtheorem{remark}[lemma]{Remark}

\parskip=\bigskipamount

\def\R{\mathbb R}

\def\r{\rtimes}
\def\t{\times}
\def\o{\otimes}
\def\e{\epsilon}

\def\ra{\rightarrow}

\def\tha{\twoheadrightarrow}

\def\a{\alpha}
\def\b{\beta}
\def\g{\gamma}

\def\s{\sigma}

\def\BE{\begin{equation}}
\def\EE{\end{equation}}

\keywords{Metaplectic groups, non-archimedean local fields, parabolic induction, principal series representations, irreducibility}

\thanks{This work has been  supported  by Croatian Science Foundation under the
project 9364.}

\date{\today}

\begin{document}

\begin{abstract} M. Hanzer and I. Mati\'c have proved in \cite{HMa2} that the genuine unitary principal series representations of the metaplectic groups are irreducible. A simple consequence of that paper 
is a criterion for the irreducibility of the non-unitary principal series representations of the metaplectic groups that we give in this paper.

\end{abstract}

\title{Two simple observations on representations of metaplectic groups}

\author[M. Tadi\'c]{Marko Tadi\'c}

\maketitle

\centerline{\sl In memory of Sibe Marde\v si\'c}

\section{Introduction}

Let $F$ be a local non-archimedean field of characteristic different from 2. 
Denote by $| \ |_F$ the normalized absolute value on $F$ and denote by $\nu$ the character $x\mapsto |x|_F$ of $F^\t=GL(1,F)$.

Maximal Levi subgroups of $GL(n,F)$ are naturally isomorphic to $GL(n_1,F) \t GL(n_2,F)$, $n_1+n_2=n$. Using this isomorphism,
J. Bernstein and A. V. Zelevinsky have introduced the notation $\pi_1\t\pi_2$ for the representations of $GL(n,F)$ parabolically induced by $\pi_1\o\pi_2$, where  $\pi_i$ is a representation of $GL(n_i,F)$ for  $i=1,2$ (see \cite{BZ2} or \cite{Z}). In the case of the symplectic groups, the maximal Levi subgroups of $Sp(2n,F)$ are naturally isomorphic to $GL(n_1,F) \t Sp(2n_2,F)$, $n_1+n_2=n$. Using this isomorphism,
one introduces the notation $\pi\r\s$ for the representation of $Sp(2n,F)$ parabolically induced by $\pi\o\s$, where $\pi$ is a representation of $GL(n_1,F)$ and $\s $ is a representation of $Sp(2n_2,F)$  (see \cite{T-Str} or \cite{T-Comp}). We shall consider central extensions
\begin{equation*}
1\longrightarrow \mu_2\hookrightarrow \widetilde G\overset{p}\tha G\longrightarrow 1
\end{equation*}
 of $G=GL(n,F)$ or $Sp(2n,F)$, where $\mu_2=\{\pm1\}$. For the description of the construction of $\widetilde G$ see the second section and the references there. 
 Let us only note that we realize $\widetilde G$ on the set $G\t\{\pm1\}$.
 A representation of $\widetilde G$ will be called genuine 
 (resp. non-genuine) 
 if $\mu_2$ acts in the representation space by a non-trivial 
 (resp. trivial) 
 character. Each representation $\pi$ of $G$ in a natural way is a non-genuine representation of $\widetilde G$, and we shall identify them in the sequel.
 
 Parabolic subgroups and Levi subgroups  in $\widetilde G$ are defined to be preimages of such subgroups in $G$ (see \cite{HMu1}). Maximal Levi subgroups in $\widetilde {GL(n,F)}$ are no more direct products of $\widetilde{GL(n_1,F)}$ and $ \widetilde{GL(n_2,F)}$, but one can still define in a natural way 
 $\pi_1\t\pi_2$ for genuine representations $\pi_i$ of $\widetilde{GL(n_i,F)}$, $i=1,2$ (see \cite{HMu1}). Analogously, one can define $\pi\r\s$ for genuine representations $\pi$ and $\s$ of $\widetilde{GL(n_1,F)}$ and  $\widetilde{Sp(2n_2,F)}$ respectively.
 
 The genuine irreducible representation of $\widetilde {Sp(0,F)}$ is denoted by $\omega_0$.

 Fix a non-trivial character of $F$ of even conductor. Define as in \cite{Ku}, page 231, the following character 
 $$
 \chi_\psi((g,\e))=\e\g(\psi_{\frac12})\g(\psi_{\frac{\det(g)}2})^{-1}
 $$
 of $\widetilde {GL(k,F)}$, where $\gamma(\eta)$ denotes the Weil index and $\psi_a(x)=\psi(ax)$.
We have $\chi_\psi^4\equiv 1$. 
Then
$
\chi_{\psi}|\mu_2\not\equiv 1,
$
and the multiplication with $\chi_\psi$ defines a bijection between genuine and non-genuine representations of $\widetilde{GL(n,F)}$. 
 
 M. Hanzer and I. Mati\'c have proved in \cite{HMa2} that the genuine unitary principal series representations  of the metaplectic groups are irreducible.
 Now the following theorem gives a criterion for the irreducibility of the genuine  non-unitary principal series representations of the metaplectic group $\widetilde{Sp(2n,F)}$:

 \begin{theorem}  Let $\xi_1, \ldots, \xi_n$ be  (not necessarily unitary) characters of $F^\t$. 
Consider the following two 
conditions:
\begin{enumerate}
\item For any $1 \leq i \leq n$, 
$$
\xi_i \neq \nu^{\pm 1/2}\xi
$$
 for any character $\xi$ of $F^\t$ satisfying $\xi=\xi^{-1}$ (i.e. of order 1 or 2).
\item For any $1 \leq i < j \leq n$, 
$$
\xi_i \neq \nu^{\pm 1} \xi^{\pm
1}_j
$$ 
(all possible combinations of two signs are allowed).
\end{enumerate}
Then the non-unitary principal series representation 
$$
(\chi_{\psi}\xi_1) \times \ldots \times
(\chi_{\psi}\xi_n) \rtimes \omega_0
$$
 of  $\widetilde{Sp(2n,F)}$ is irreducible if and only if conditions (1) and (2) hold.  
\end{theorem}

 We expect that the above theorem must be well-known to the specialists in the field. Nevertheless, since we could not find a written reference, we present a proof of this simple fact here.
 
 In this paper we also give a formula related to computation of the Jacquet modules in the case of symplectic and metaplectic groups. 
 
 Discussions and questions of E. Lapid motivated us to think about the above irreducibility criterion. The irreducibility in the above theorem behaves in the same way as the irreducibility in the case of the split odd-orthogonal groups (see Remark \ref{rem}). A criterion for the irreducibility in the  non-genuine case, i.e. for  $Sp(2n,F)$, can be found in \cite{T-Comp}. The proof there does not apply to the case that we consider in this paper. Namely, we do not have here the theory of $R$-groups. Instead of that, we  apply  the main result of \cite{HMa2} which we have already mentioned. 
 
 We are very thankful to M. Hanzer for discussions and explanations regarding representation theory of metaplectic groups, and for  reading  the paper and corrections. 
 
The content of the paper is as follows. In the second section we introduce the notation used in the paper. The third section brings the proof of the above theorem while the last section   proves a  formula relating computation of Jacquet modules in the case of symplectic and metaplectic groups.

\section{Notation and preliminary results}

In this paper we shall consider central extensions
\begin{equation}
\label{G}
1\longrightarrow \mu_2\hookrightarrow \widetilde G\overset{p}\tha G\longrightarrow 1
\end{equation}
 of a $p$-adic general linear linear or a $p$-adic symplectic  group $G$, where $\mu_2=\{\pm1\}$. The group $\widetilde G$  has a natural topology of a locally compact totally disconnected group (see \cite{HMu1}). By a representation in this paper we shall always mean a smooth representation. A representation of $\widetilde G$ will be called genuine (resp. non-genuine) if $\mu_2$ acts in the representation space by a non-trivial (resp. trivial) character.

 The Grothendieck group of the category of all the smooth representations of $G$  will be denoted by $\mathfrak R(G)$.  Further, the  Grothendieck group of the category of all the genuine smooth representations of $\widetilde G$ (resp. non-genuine smooth representations of $\widetilde G$) will be denoted by $\mathfrak R(\widetilde G)^{gen}$ (resp. $\mathfrak R(\widetilde G)^{non-gen}$). Since each non-genuine representation of $\widetilde G$ factors uniquely to a representation of $G$, and conversely, each representation of $G$ in a natural way extends to a non-genuine representation of $\widetilde G$, we identify the  non-genuine representations of $\widetilde G$  with the representations of $G$.
 We  identify in a natural way $\mathfrak R(\widetilde G)^{non-gen}$ with $\mathfrak R( G)$.
 
Fix a $p$-adic field $F$ of characteristics different from 2. The normalized absolute value of $F$ will be denoted by $|\ |_F$. We denote by $\nu$ the character of $F^\t=GL(1,F)$ given by $x\mapsto |x|_F$.

For $GL(n):=GL(n,F)$, we realize $\widetilde {GL(n)}$ on the set $GL(n)\t\mu_2$ with the multiplication defined by $(g_1,\e_1)(g_2,\e_2)=(g_1g_2,\e_1\e_2 (\det(g_1),\det(g_2))_F)$, where $(\ \ ,\ \ )_F$ denotes the Hilbert symbol. 
In analogous way as in \cite{BZ2}, in \cite{HMu1} is introduced the multiplication $\t$ among genuine representations of $\widetilde{GL(n_1)}$ and $\widetilde{GL(n_2)}$.
 The sum of all
$\mathfrak R(\widetilde {GL(n)})^{gen}$, $n\geq 0$, will be denoted by $R^{gen}$. 
One can factor the multiplication through $R^{gen}\o R^{gen}$, and the obtained mapping is denoted by $m_\sim:R^{gen}\o R^{gen} \ra R^{gen}$. In analogous way as in \cite{BZ1}, in \cite{HMu1} is defined  the comultiplication, denoted  $m^*_\sim: R^{gen}\ra R^{gen}\o R^{gen}$. The transposition map on 
$R^{gen}\o R^{gen}$ (which switches the tensor factors) is denoted by $\kappa$.

In the case of symplectic group $Sp(2n):=Sp(2n,F)$ of rank $n$, $\widetilde{Sp(2n,F)}$ equals as a set to $Sp(2n,F)\t\mu_2$, while the multiplication is given by Rao's cocycle $c_{Rao}$ (\cite{R}) with the formula $(g_1,\e_1)(g_2,\e_2)=(g_1g_2,\e_1\e_2 c_{Rao}(g_1,g_2))$. 
For  $0\leq k\leq n$, there is a  parabolic subgroup\footnote{It is maximal parabolic subgroup except for $k=0$.} $P_{(k)}=M_{(k)}\ltimes N_{(k)}$ standard with respect to the upper triangular matrices in $Sp(2n)$, whose Levi factor $M_{(k)}$ is naturally isomorphic to $GL(k)\t Sp(2(n-k))$ (see \cite{HMu1}). Now parabolic subgroup $\widetilde P_{(k)}$ is defined to be the preimage $p^{-1}(P_{(k)})$. With $\widetilde M_{(k)} =p^{-1}(M_{(k)})$ we have $\widetilde P_{(k)}=\widetilde M_{(k)}\ltimes  N_{(k)}'$, where $N_{(k)}'=N_{(k)}\t\{1\}$.

   Now there is a natural epimorphism 
  \begin{equation}
  \label{epi}
  \widetilde {GL(k)} \t \widetilde {Sp(2(n-k))}\tha \widetilde M_{(k)}
  \end{equation}
(see \cite{HMu1}).   If $\pi$ and $\sigma$ are both genuine (or both non-genuine) representations of $\widetilde {GL(k)}$ and $ \widetilde {Sp(2(n-k))}$ respectively, then $\pi\o\sigma$ factors to a representation of $\widetilde M_{(k)}$, and the representation of $\widetilde{Sp(2n,F)}$ parabolically induced by $\pi\o\sigma$ from $\widetilde P_{(k)}$ is denoted by
  $$
  \pi\r\s.
  $$
 Define
$$
\a((g,\e)):=(\det(g),-1)_F.
$$
Then this is a character of $\widetilde {GL(k)}$ and   in the Grothendieck group holds
 \begin{equation}
  \label{tilde}
\pi\r\s=(\a\cdot \tilde\pi)\r\s.
 \end{equation}

 Further, the sum of all
  $\mathfrak R(\widetilde {Sp(n)})^{gen}$, $n\geq 0$,  is denoted by $R_S^{gen}$. Then $\rtimes$ factors in a natural way to
$\rtimes: R^{gen}\o R_S^{gen} \rightarrow R_S^{gen} $. For an irreducible genuine representation $\pi$ of  $\widetilde {Sp(n)}$, 
the Jacquet module with respect to $\widetilde P_{(k)}=\widetilde M_{(k)}\ltimes  N_{(k)}'$ can be pulled back to a representation of $\widetilde {GL(k)} \t \widetilde {Sp(2(n-k))}$ using epimorphism \eqref{epi}. Its  semisimplification can be considered as an element of $R^{gen}\o R_S^{gen}$. Now the sum of  all these semisimplifications for $0\leq k\leq n$ is denoted by
$$
\mu^*(\pi)\in R^{gen}\o R_S^{gen}.
$$
 We  extend  additively  $\mu^*$ to a mapping $\mu^*: R_S^{gen} \rightarrow R^{gen}\o R_S^{gen} $ (see details in \cite{HMu1}).

 Now  we shall recall of an important formula of  \cite{HMu1}  
\begin{equation}
\label{mu-met}
\mu^*(\pi\r\s)= M^*_\sim(\pi)\r\mu^*(\s),
\end{equation}
where
$$
M^*_\sim=(m_\sim\o id) \circ (\a\cdot \tilde{\ }\o m^*_\sim)\o\kappa\o m^*_\sim
$$
(here $\tilde{\ }$ denotes the contragredient mapping).

 Fix a non-trivial character $\psi$ of $F$ of even conductor. Define as in \cite{Ku}, page 231, the following character 
 $$
 \chi_\psi((g,\e))=\e\g(\psi_{\frac12})\g(\psi_{\frac{\det(g)}2})^{-1}
 $$
 of $\widetilde {GL(k)}$, where $\gamma(\eta)$ denotes the Weil index and $\psi_a(x)=\psi(ax)$.
We have $\chi_\psi^4\equiv 1$. 

The following three facts  are essential for us
$$
\a=\chi_{\psi}^2,
$$
$$
\chi_{\psi}|\mu_2\not\equiv 1,
$$
and for  representations $\pi_i$ of $\widetilde{GL(n_i)}$ holds
\begin{equation}
\label{chi-psi}
\chi_{\psi}(\pi_1\t\pi_2)\cong (\chi_{\psi}\pi_1)\t (\chi_{\psi}\pi_2)
\end{equation}
when both $\pi_i$ have  the  same restriction to $\mu_2$ (\cite{HMu1}).

Observe that the multiplication with $\chi_\psi$ defines an isomorphism of $R\ra R^{gen}$, which is also denoted by $\chi_\psi$.

For representations $\pi_i$ of $\widetilde{GL(n_i)}$, $1\leq i\leq 3$ and a representation $\s$ of $\widetilde{Sp(2n)}$ holds
\begin{equation}
\label{asso-GL}
\pi_1\t(\pi_2\t\pi_3)\cong (\pi_1\t\pi_2)\t\pi_3,
\end{equation}
and
\begin{equation}
\label{asso-Sp}
\pi_1\r(\pi_2\r\s)\cong (\pi_1\t\pi_2)\r\s.
\end{equation}
The above two isomorphisms follow from the general facts proved about representations of $l$-groups in \cite{BZ2}.

\section{Non-unitary principal series representations of metaplectic groups}

Each character $\phi$ of $F^\t$  can be written in a unique way  as
$$
\phi=\nu^{e(\phi)}\phi_u,
$$
where $e(\phi)\in\R $ and $\phi_u$ is unitary.

\begin{theorem}  
\label{theorem}
Let $\xi_1, \ldots, \xi_n $ be  (not necessarily unitary) characters of $F^\t$. 
Consider the following two 
conditions:
\begin{enumerate}
\item For any $1 \leq i \leq n$, 
$$
\xi_i \neq \nu^{\pm 1/2}\xi
$$
 for any character $\xi$ satisfying $\xi=\xi^{-1}$ (i.e. of order 1 or 2).
\item For any $1 \leq i < j \leq n$, 
$$
\xi_i \neq \nu^{\pm 1} \xi^{\pm
1}_j
$$ 
(all possible combinations of two signs are allowed).
\end{enumerate}
Then the non-unitary principal series representation 
$$
(\chi_{\psi}\xi_1) \times \ldots \times
(\chi_{\psi}\xi_n) \rtimes \omega_0
$$
 of  $\widetilde{Sp(2n,F)}$ is irreducible if and only if conditions (1) and (2) hold.  
\end{theorem}

\begin{proof}
For the proof, first observe that in the Grothendieck group we have
$$
(\chi_{\psi}\xi_i) \rtimes \omega_0=\a (\chi_{\psi}^{-1}\xi_i^{-1}) \rtimes \omega_0=(\chi_{\psi}\xi_i^{-1}) \rtimes \omega_0,
$$
i.e.
\begin{equation}
\label{-}
(\chi_{\psi}\xi_i) \rtimes \omega_0=(\chi_{\psi}\xi_i^{-1}) \rtimes \omega_0.
\end{equation}

From this and the property $\chi_{\psi}(\pi_1\t\pi_2)\cong (\chi_{\psi}\pi_1)\t (\chi_{\psi}\pi_2)$ that we have mentioned earlier, follow easily  that all the following elements
\begin{equation}
\label{all}
(\chi_{\psi}\xi_{p(1)}^{\epsilon_1}) \times \ldots \times
(\chi_{\psi}\xi_{p(n)}^{\epsilon_n}) \rtimes \omega_0
\end{equation}
define the same element of the Grothendieck group, for any  permutation $p$ of $\{1,\dots,n\}$ and any choice of  $\epsilon_i\in\{\pm\}$.

Recall that 
$$
(\chi_{\psi}\xi_i) \rtimes \omega_0
$$
 reduces if and only if $\xi_i = \nu^{\pm 1/2}\xi
$
 for some character $\xi$ satisfying $\xi=\xi^{-1}$ (this follows for example from \cite{HMu2}; it is explicitly written in \cite{HMa1}).
Now  this and the reducibility in the $\widetilde {GL(2)}$ case, directly imply that if (1) or (2) does not hold, we have reducibility.

We are going now to prove the irreducibility. Suppose that (1) and (2) hold. We shall do it by induction. For $n=1$ the claim obviously holds. Consider  $n>1$ and suppose that the irreducibility holds for smaller indexes.

Now \eqref{-}, \eqref{asso-GL}  and \eqref{asso-Sp} imply that all the representations \eqref{all} are not only equal in the Grothendieck group, but they are all isomorphic.

Observe that since for  all the choices of $\e_i$ in \eqref{all} we get the same element in the Grothendieck group, it is enough to prove the irreducibility for one of them. In this way we get that it is enough to prove irreducibility in the case
$$
e(\xi_1)\geq \dots \geq e(\xi_n)\geq 0.
$$
We shall assume this below.

If $e(\xi_1)=0$, \cite{HMa2} implies the irreducibility. Therefore we need to consider the case $e(\xi_1)>0$. Take the biggest $i$ such that $e(\xi_i)$ is still $>0$. Clearly, $i\geq 1$.

The formula \eqref{mu-met} 
and the fact that $M^*_\sim(\chi_\psi)= \chi_\psi\xi\o1+\chi_\psi\tilde\xi\o1+1\o \chi_\psi\xi$
 imply that 
\begin{equation}
\label{jmsq}
[(\chi_{\psi}\xi_1) \times \ldots \times (\chi_{\psi}\xi_i)] \o [(\chi_{\psi}\xi_{i+1})\t\dots\t
(\chi_{\psi}\xi_n) \rtimes \omega_0]
\end{equation}
has multiplicity one in the Jacquet module of $(\chi_{\psi}\xi_1) \times \ldots \times
(\chi_{\psi}\xi_n) \rtimes \omega_0$. Note that \eqref{jmsq} is irreducible since we know that the left hand side tensor factor is irreducible from the case of general linear groups (applying \eqref{chi-psi}), and right hand factor is irreducible by the inductive assumption. This implies that the last representation has unique irreducible subrepresentation, and that the irreducible subrepresentation has multiplicity one in the whole induced representation. 

Let $\pi$ be an irreducible subquotient of $(\chi_{\psi}\xi_1) \times \ldots \times
(\chi_{\psi}\xi_n) \rtimes \omega_0$. 
Then $\pi$ is not cuspidal (the proof  goes analogously as the proof of (d) in Theorem in section 2.4 of \cite{BZ2}).
Using this, and the formula for $\mu^*((\chi_{\psi}\xi_1) \times \ldots \times
(\chi_{\psi}\xi_n) \rtimes \omega_0)$, we get that a Jacquet module of $\pi$ has some 
$
(\chi_{\psi}\xi_{p(1)}^{\epsilon_1}) \otimes \ldots \otimes
(\chi_{\psi}\xi_{p(n)}^{\epsilon_n}) \otimes \omega_0
$
for a quotient, with $\e_i$ and $p$ as in \eqref{all}. Now the Frobenius reciprocity implies that $\pi$ embeds into some of the representations \eqref{all}. Since all the irreducible  subrepresentations for the family \eqref{all} are  isomorphic, and since  that subrepresentation has multiplicity one, we get irreducibility.
\end{proof}

\begin{remark}
\label{rem} Consider now split odd-orthogonal group $SO(2n+1;F)$. We shall use the notation introduced for these groups in \cite{T-Str}. Let $\xi_1, \ldots, \xi_n $ be  (not necessarily unitary) characters of $F^\t$. 
Then the non-unitary principal series representation 
$$
\xi_1 \times \ldots \times
\xi_n \rtimes 1_{SO(1)}
$$
 of the split special orthogonal group $SO(2n+1,F)$ is irreducible if and only if conditions 
(1) and (2) from Theorem \ref{theorem} hold for $\xi_1,\dots,\xi_n$.  

Therefore the representation $
\xi_1 \times \ldots \times
\xi_n \rtimes 1_{SO(1)}
$ of $SO(2n+1,F)$ is irreducible if and only if the representation $
(\chi_{\psi}\xi_1) \times \ldots \times
(\chi_{\psi}\xi_n) \rtimes \omega_0
$ of $\widetilde {Sp(2n)}$ is irreducible.
\end{remark}

\section{Linear and metaplectic cases}

In \cite{Z} is introduced a Hopf algebra constructed on the sum $R$ of $\mathfrak R(GL(n,F)), n\geq0$. 
The multiplication and comultiplication are denoted there by $m$ and $m^*$ respectively.

A multiplication   $\r$ between representations of general linear and symplectic groups is  introduced in \cite{T-Str}. There is proved that
$$
\mu^*(\pi\r\s)= M^*(\pi)\r\mu^*(\s),
$$
where
$$
M^*=(m\o id) \circ ( \tilde{\ }\o m^*)\o\kappa\o m^*
$$

The formula for $\mu^*(\tau\r\s)=M^*_\sim(\tau)\r\mu^*(\s)$, where $\tau$ and $\s$ are genuine representations of corresponding covering groups, was very useful in the case of metaplectic groups. We can write the genuine representation $\tau$  as
$$
\tau=\chi_\psi\pi,
$$
where
$\pi$ is a representation of a general linear group (i.e. a non-genuine representation of the covering group).
Now we shall write a very simple formula relating $M^*_\sim(\chi_\psi\pi)$ and $M^*(\pi)$:

\begin{lemma} We have
$$
 M^*_\sim=( \chi_{\psi}\o \chi_{\psi}) \circ M^* \circ \chi_{\psi}^{-1}.
$$
In other words, if $\pi $ is
 a finite length  representation  of ${GL(n)}$ and if we write 
 $$
 M^*(\pi)=\sum_i \b_i\o\g_i,
 $$
  then 
  $$
  M^*_\sim(\chi_\psi\pi)=\sum_i \chi_\psi\b_i\o\chi_\psi\g_i.
  $$
\end{lemma}

\begin{proof}
First the formula
$
\chi_{\psi}(\pi_1\t\pi_2)\cong (\chi_{\psi}\pi_1)\t (\chi_{\psi}\pi_2)
$
for representations $\pi_i$ of $GL(n_i)$ implies
$$
 m_\sim \circ (\chi_{\psi}\o \chi_{\psi})=  \chi_{\psi} \circ m.
$$

Consider the projection $p:\widetilde{GL(n)}\ra GL(n)$, and the standard Levi subgroup $M=:GL(n_1)\t GL(n_2)$ in $GL(n)$, where $n_1+n_2=n$. Let $\widetilde M$ be the preimage of $M$, and denote by $\phi: \widetilde M\ra M$ the restriction of the projection $p$.
Before Proposition 4.1 in \cite{HMu1} it is observed that the equality
$
(\chi_{\psi}|\tilde M)\circ \phi= \chi_{\psi} \o \chi_{\psi}
$
holds.

Now one directly sees that the following dual relation holds
$$
 m^*_\sim\circ \chi_{\psi}=(\chi_{\psi}\o \chi_{\psi})\circ m^*.
$$

Let now  $\pi$ be a finite length representation of $GL(k)$ (i.e. a non-genuine representation of $\widetilde{GL(k)}$).
  Consider
$$
 M^*_\sim(\chi_{\psi}\pi) = \big( m_\sim \o id\big)\circ \big(\tilde{\ }\cdot\a\o m^*_\sim\big)\circ \kappa \circ  m^*_\sim\circ \chi_{\psi}(\pi)
$$
$$
= \big( m_\sim \o id\big)\circ \big(\tilde{\ }\cdot \a\o m^*_\sim\big)\circ \kappa \circ \big(\chi_{\psi}\o \chi_{\psi}\big)\circ m^*(\pi)
$$
$$
= \big( m _\sim\o id\big)\circ \big(\tilde{\ }\cdot \a\o m^*_\sim\big) \circ \big(\chi_{\psi}\o \chi_{\psi}\big)\circ \kappa\circ m^*(\pi)
$$
$$
= \big( m_\sim \o id\big)\circ \big((\tilde{\ }\cdot\a \circ \chi_{\psi})\o  (m^*_\sim \circ  \chi_{\psi})\big)\circ \kappa\circ m^*(\pi)
$$
$$
= \big( m_\sim \o id\big)\circ \big(\a  \chi_{\psi}^{-1}\cdot \tilde{\ }\o((\chi_{\psi}\o \chi_{\psi})\circ m^*)\big)\circ \kappa\circ m^*(\pi)
$$
$$
= \big( m_\sim \o id\big)\circ \big( \chi_{\psi}\cdot\tilde{\ }\o(\chi_{\psi}\o (\chi_{\psi})\circ m^*)\big)\circ \kappa\circ m^*(\pi)
$$
$$
= \big( m_\sim \o id\big)\circ \big ( \chi_{\psi}\o\chi_{\psi}\o \chi_{\psi}) \circ \big( \tilde{\ }\o id \o id\big)\circ \big(id\o m^*\big)\circ \kappa\circ m^*(\pi)
$$
$$
= \big( (m_\sim \circ  ( \chi_{\psi}\o\chi_{\psi}))\o \chi_{\psi}\big) \circ\big ( \tilde{\ }\o id \o id\big)\circ \big(id\o m^*\big)\circ \kappa\circ m^*(\pi)
$$
$$
= \big(( \chi_{\psi}\circ m)\o \chi_{\psi}\big) \circ \big( \tilde{\ }\o id \o id\big)\circ \big(id\o m^*\big)\circ \kappa\circ m^*(\pi)
$$
$$
=  \big( \chi_{\psi}\o \chi_{\psi}\big) \circ \big(  m\o id) \circ ( \tilde{\ }\o id \o id\big)\circ \big(id\o m^*\big)\circ \kappa\circ m^*(\pi)
$$
$$
= \big ( \chi_{\psi}\o \chi_{\psi}\big) \circ \big(  m\o id\big) \circ \big( \tilde{\ }\o  m^*\big)\circ \kappa\circ m^*(\pi)
$$
$$
= \big ( \chi_{\psi}\o \chi_{\psi}\big) \circ M^*(\pi),
$$
i.e. we have proved that 
$$
 M^*_\sim(\chi_{\psi}\pi) =\big( \chi_{\psi}\o \chi_{\psi}\big) \circ M^*(\pi).
$$
\end{proof}

\end{document}